\documentclass[12pt,reqno]{amsart}
\usepackage{amsfonts,color,amsmath,amssymb,fancyhdr}
\usepackage{amsthm}
\usepackage{txfonts}
\usepackage{float}
\usepackage{multirow}
\def\Box{\vcenter{\vbox{\hrule\hbox{\vrule
     \vbox to 8.8pt{\hbox to 10pt{}\vfill}\vrule}\hrule}}}

\newtheorem{thm}{Theorem}[section]
\newtheorem{lemma}[thm]{Lemma}
\newtheorem{corollary}[thm]{Corollary}
\newtheorem{prop}[thm]{Proposition}
\newtheorem{definition}[thm]{Definition}
\newtheorem{example}[thm]{Example}
\numberwithin{equation}{section}
\newtheorem{remark}[thm]{Remark}

\newcommand{\bC}{\mathbb C}
\newcommand{\bF}{\mathbb F}

\newcommand{\bZ}{\mathbb Z}

\definecolor{Purple}{rgb}{0.5,0,0.5}

\newcommand{\tup}{\textup}

\begin{document}

\title{On central difference sets in Suzuki $p$-groups of type $A$}
\begin{center}
\author{Wendi Di,  Zhiwen He$^*$ }
\end{center}
\address{School of Mathematical Sciences, Zhejiang University, Hangzhou 310027,  China}
\email{Wendyjj@zju.edu.cn}
\address{School of Mathematical Sciences, Zhejiang University, Hangzhou 310027,  China}
\email{zhiwenhe94@163.com}

\begin{abstract}
In this paper, when  the order of $\theta$ is even, we prove that there exists no central difference sets in $A_2(m,\theta)$ and establish some non-existence results of central partial difference sets in $A_p(m,\theta)$ with $p>2$. When the order of $\theta$ is odd, we construct central  difference sets in $A_2(m,\theta)$. Furthermore, we give some reduced linking systems of difference sets in $A_2(m,\theta)$ by using the difference sets we constructed. In the case $p>2$, we construct Latin square type central partial difference sets in $A_p(m,\theta)$ by a similar method.

\end{abstract}
\keywords{ central difference sets; central partial difference sets;  characters; linking systems.\\
{\bf  Mathematics Subject Classification (2010) 20C15, 20D15, 20E45, 05B10}\\
{\bf  Funding information: National Natural Science Foundation of China under Grant No. 11771392.}\\
$^*$Correspondence author}
\maketitle

\section{Introduction}\label{Introduction}
\subsection{Difference sets}\quad

  Let $G$ be a  group of order $v$, and let $D$ be a subset of $G$ with cardinality $k$. Then we call $D$ a \emph{$(v,k,\lambda,n)$-difference set}  in $G$ with $n=k-\lambda$ provided that the expressions $d_1d_2^{-1}$, for $d_1,d_2\in D$ with $d_1\neq d_2$ represent each non-identity element in $G$ exactly  $\lambda$ times. Particularly, we say the set $D$  a \emph{central difference set} if it is a union of conjugacy classes in $G$.

In 1938, Singer \cite{Singer} first introduced difference sets  in cyclic groups  in the study of the regular automorphism groups of projective geometries. After that, numerous papers  have been published on this subject and many of them used algebraic number theory, group theory, finite geometry, and representation and character theory to establish constructive and non-existence results (see, for example \cite{Davis2,Davis,Feng4,Gow,R. Kraemer,Turyn}). Standard introductions to difference sets occur in  \cite{Jungnickel1,Jungnickel2}.

Most of the progress in the research of difference sets has occurred in the Hadamard difference sets (see, for example \cite{AbuGhneim,Davis2,Ding2,Ding,Feng3,R. Kraemer,Smith}).
A difference set is called a \emph{Hadamard (or, alternatively, Menon) difference set} if it has parameters $(4t^2,2t^2-t,t^2-t,t^2)$.
Theorem \ref{thm_2groups} shows that the parameters of all non-trivial difference sets in $2$-groups must take this common form.
\begin{thm}\cite[Theorem II 3.17]{Beth}\label{thm_2groups}
  Suppose a group $G$ of  order $2^r$ contains a $(v,k,\lambda,n)$-difference set where $2\leq k\leq v/2$. Then $r=2(d+1)$ for some non-negative integer $d$, and \[(v,k,\lambda,n)=(2^{2(d+1)},2^{d}(2^{d+1}-1),2^{d}(2^{d}-1), 2^{2d}).\]
\end{thm}
The central objective on difference sets is to determine which groups contain a difference set.
The existence of difference sets in abelian $2$-groups was completely answered by Kraemer \cite{R. Kraemer} and Jedwab \cite{J.Jedwab} that an abelian $2$-group with order $2^{2d+2}$ has a difference set if and only if the exponent of the group is no more than $2^{d+2}$.
 However, this is not true in the non-abelian case. For instance,  \cite {Davis2,Davis} gave infinite families of difference sets in groups with high exponent.
 Dillon \cite{Dillon} generalized a construction provided by McFarland which could be used in non-abelian $2$-groups.
  \begin{lemma}\cite[Dillon's constructions]{Dillon}\label{Lem_DIllon}
 Let $q$ be a prime power and $d$ a non-negative integer, and let $s=\frac{q^{d+1}}{q-1}$.
Let $G$ be a group containing a central subgroup $E$ of index $s+1$, which is isomorphic to the elementary abelian group of order $q^{d+1}$.
Let $g_0,g_1,\cdots, g_s$ be a set of coset representatives for $E$ in $G$. Let $H_1,H_2,\cdots, H_s$ be the subgroups of $G$ corresponding to the hyperplanes of $E$, under an isomorphism $\phi$, when $E$ is regarded as a vector space of  dimension $d+1$ over $\bF_q$.
Then $D=\cup_{i=1}^s g_iH_i$
 is a difference set  in $G$ with McFarland parameters
 \[(v,k,\lambda,n)=\big(q^{d+1}(s+1), q^{d}s, q^{d}(s-q^{d}), q^{2d}).\]
\end{lemma}
As for the non-existence in $2$-groups, there were two known  results established by Turyn \cite{Turyn} and Ma \cite{Ma2}, respectively.
 In this paper, we answer the question that whether a central Hadamard difference set exists in the Suzuki $2$-groups of type $A$, a family of non-abelian groups.

\subsection{Linking systems of difference sets}\quad

A linking system of difference sets is a collection of group difference sets, which is first introduced by Davis et al \cite{linking system1}.
\begin{definition}\cite{linking system2}
  Let $G$ be a group of order $v$, and let $l\geq 2$.  Suppose $\mathcal{L}=\{D_{i,j}:0\leq i,j\leq l \tup{ and } i\neq j\}$ is a collection of size $l(l+1)$ of $(v,k,\lambda,n)$-difference sets in $G$. Then, $\mathcal{L}$ is a $(v,k,\lambda,n,l+1)$-linking system of difference sets in $G$ if there are integers $\mu,\eta$ such that for all distinct $i,j,h$, the following equations hold in $\bZ[G]$:
  \[
  D_{h,i}D_{i,j}=(\mu-\eta)D_{h,j}+\eta G,\, D_{i,j}=D_{j,i}^{(-1)}.
  \]
\end{definition}
A linking system of difference sets gives rise to a system of linked symmetric designs, as introduced by Cameron \cite{Cameron},  which is exactly equivalent to a $3$-class $Q$-antipodal cometric association scheme \cite{Dam}.
The central problems are to determine which groups contain  a linking system of difference sets, and how large such a system can be.
Almost all previous constructive results for linking systems of difference sets were in $2$-groups.
Jonathan et al. \cite{linking system2} showed that neither the McFarland/Dillon nor the Spence construction of difference sets can give rise to a linking system of difference sets in non-$2$-groups.
It remains an important open question that whether a linking system of difference sets can exist in non-$2$-groups.

The first examples of linking systems were found in \cite{Cameron2} in the context of \emph{bent sets}.
Jonathan et al. \cite{linking system2} gave a new construction for linking systems of difference sets in $2$-groups, taking advantage of a connection with group difference matrices.
Here, we give a new construction for linking systems of difference sets in non-abelian $2$-groups by using group characters.
\subsection{Partial difference sets}\quad

Let $G$ be a  group of order $v$, and let $D$ be a $k$-subset of $G$. If the differences $d_1d_2^{-1}$  for $d_1,d_2\in D, d_1\neq d_2$ contain  every non-identity element of $D$  exactly $\lambda$ times  and every non-identity element of $G-D$  exactly $\mu$ times, then $D$ is called a \emph{$(v,k,\lambda, \mu )$-partial difference set}  in $G$. In particular, when the identity $1_G\notin D$ and $D^{(-1)}=D$, we call $D$ is  \emph{regular}. Besides, if $D$ is a union of conjugacy classes of $G$, then $D$ is called a \emph{central partial difference set}.  A $(v,k,\lambda, \mu )$-partial difference set with $\lambda=\mu$ is an ordinary $(v,k,\lambda)$-difference set.

A regular partial difference set $D$ in a finite  group $G$ corresponds to a strong regular Cayley graph $\tup{Cay}(G,D)$.
Partial difference sets have close connections with other branches of combinatorics as well as coding theory and finite geometry. There were many partial difference sets constructed  from projective two-weight codes and projective two-intersecting sets in \cite{Calderbank}.
 The survey of Ma \cite{Ma} is an excellent reference for partial difference sets, which provides a thorough survey on the development of the subject up to 1994.
 The case where $G$ is abelian has been studied extensively.
 However, there are very few infinite families of (regular) partial difference sets in non-abelian groups. The only known constructions of such partial difference sets, to our best knowledge, are those in \cite{Ghinelli,Swartz,Tao1}.

In this paper, we consider   partial difference set with a parameter $(n^2,r(n-\varepsilon),r^2 +\varepsilon(n-3r),r^2-\varepsilon r)$ for $\varepsilon=\pm 1$. The partial difference set with such a parameter set is called a \emph{Latin square  type partial  difference set} if $\varepsilon=1$ and a \emph{negative Latin square  type partial difference set}  if $\varepsilon=-1$.
There are many constructions of partial difference sets with both parameter sets (see \cite{Chee,Chen,Davis&Xiang,Hamilton,Tan}). The known partial difference sets  with negative Latin square parameters are relatively rare. Almost all the known groups that contain Latin square or negative Latin square type partial difference sets are abelian $p$-groups and many of these are obtained from quadratic forms \cite{Hamilton}, bent functions \cite{Chee,Chen,Tan}, and various other combinatorial objects. Most notably, Davis and Xiang \cite{Davis&Xiang} construct the first known family of negative Latin square type partial difference sets in non-elementary abelian $2$-groups of exponent $4$ by using quadrics.
Here, we get two kinds of Latin square type central partial difference sets in the non-abelian Suzuki $p$-groups  $A_p(m,\theta)$ with $p>2$.

We conclude this section by giving the framework of the remaining of this article: In Section \ref{preliminary}, we review results on difference sets, partial difference sets and linking systems that we shall use later. Besides, we introduce the Suzuki $p$-groups of type $A$ and list their character tables.
In Section \ref{nonexistence}, we consider the case where $o(\theta)$ is even and establish the non-existence results  in $A_p(m,\theta)$.
In Section \ref{constructions}, we consider the case where $o(\theta)$ is odd,  construct  central difference sets in $A_2(m,\theta)$, and
 give linking systems of difference sets in $A_2(m,\theta)$.
Similarly, when $p>2$, we construct  Latin square type  central partial difference sets in $A_p(m,\theta)$.

\section{Preliminaries}\label{preliminary}
\subsection{ Difference sets and partial difference sets}\quad

 Let $G$ be a finite multiplicative group and consider the group ring $\bZ[G]$. If $D$ is a subset of $G$, we will abuse notation by writing $D$ as an element of $\bZ[G]$, i.e., $D=\sum_{d\in D}d$. We write $D^{(-1)}$ for the group ring element $\sum_{d\in D}d^{-1}$.
Let $A,B \in\bZ[G]$, then $ AB^{(-1)}=\sum_{a\in A,b\in B}ab^{-1}$. The definitions of difference sets and partial difference sets in  $G$ immediately yield the following equivalent definitions by using group ring.
 \begin{lemma}\label{lem_DD^-1}
Let $G$ be a group of order $v$ and let $D$ be a subset of $G$ with cardinality $k$. Then
\begin{enumerate}
  \item $D$ is a $(v,k,\lambda)$-difference set in $G$ if and only if
\[
DD^{(-1)}=\lambda G+(k-\lambda )1_{G} .
\]

  \item  $D$ is a regular  $(v,k,\lambda,\mu)$-partial difference set in $G$ if and only if
\[
DD^{(-1)}=\mu G+(\lambda -\mu)D+ (k-\mu) 1_G .
\]

\end{enumerate}
 \end{lemma}
  \begin{lemma}\cite{Ma2}\label{D-{1}}
   If $D$ is a $(v,k,\lambda,\mu)$-partial difference set in a finite group $G$ and $\lambda\neq \mu$, then $D^{(-1)}=D$. Besides, if $1_G\in D$, then $G\setminus D$ and $D\setminus\{1_G\}$ are also partial difference sets in $G$.
 \end{lemma}
 Thus we always consider the regular partial difference sets.
\begin{lemma}\label{thm_pdsds}
Let $G$ be a group of order $v$ and let $D$ be a subset of $G$ with cardinality $k$. Then
\begin{itemize}
\item[(1)]\cite[Theorem 2.1]{Davis} $D$ is a $(v,k,\lambda)$-difference set in $G$ if and only if \[\Phi(D)\Phi(D^{(-1)})=(k-\lambda)I\] for every non-trivial  irreducible representation $\Phi$ of $G$.
\item[(2)] $D$ is a regular $(v,k,\lambda,\mu)$-partial difference set in $G$ if and only if \[\Phi(D)^2=(\lambda-\mu)\Phi(D)+(k-\mu)I\]for every non-trivial irreducible representation $\Phi$ of $G$.
\end{itemize}
\end{lemma}
\begin{proof}
 It is known that a subset $D$ of $G$ is completely determined by its image under  the regular representation of $G$. Moreover, the regular representation is completely determined by all the  irreducible representations of $G$. Thus the subset $D$ is completely determined by its image under all the irreducible representations.

 We claim that $\Phi(G)=0$ for any non-trivial irreducible representation $\Phi$ of $G$. Note that $\Phi(h)\Phi(G)=\Phi(G)\Phi(h)$ for any $h\in G$, from Schur's Lemma  we have $\Phi(G)=c I$ for some $c\in \bC$.
 Since $\Phi$ is non-trivial, we can choose $g\in G$ such that $\Phi(g)\neq I$.

Since $\Phi(g)\Phi(G)=\Phi(G)$, we have $c(\Phi(g)-I)=0$. Then we have $c=0$ from the choice of $g$. The results then follow from Lemma \ref{lem_DD^-1}.
\end{proof}

\begin{lemma}\cite[Chapter 12]{character theory}\label{lem_Z(CG)}
  Let $G$ be a finite group. Suppose $C_1,C_2,\cdots,C_s$ are all the distinct conjugacy classes of $G$. Then the group ring elements  $C_1,C_2,\cdots,C_s$  form a basis of the centre of the group algebra $Z(\bC[G])$.
 \end{lemma}

We denote the set of irreducible characters of a group $G$ by $\tup{Irr}(G)$ and denote the set of non-trivial irreducible characters of $G$ by $\tup{Irr}(G)^*$.
\begin{definition}\label{def_omega}
Suppose $\{C_i: i=1,2, \cdots, s\}$ is the set of all the distinct conjugacy classes of $G$ and $g_i$ is a representative element of $C_i, i=1,2,\cdots,s$.
For each $\chi\in \tup{Irr}(G)^*$,  define
 \[\omega_{\chi} :Z(\bC[G])\rightarrow \bC, \,\, C_i\mapsto \frac{\chi(C_i)}{\chi(1)}=\frac{|C_i|\chi(g_i)}{\chi(1)}, \, i=1,2,\cdots, s.\]
 \end{definition}\label{wx}
 From Lemma \ref{lem_Z(CG)}, we get that for each $\chi\in\tup{Irr}(G)^*$, $\omega_{\chi}$  is a homomorphism from  $Z(\bC [G])$ to $\bC$.

\begin{thm}\label{char theorem}
Let $G$ be a  group of order $v$ and  $D$ be a union of some conjugacy classes of $G$ with size $k$.
Let $\omega_\chi$ be defined in Definition \ref{def_omega} for each $\chi\in \tup{Irr}(G)^*$.
Then the following holds:
\begin{itemize}
\item[(1)] $D$ is a central $(v,k,\lambda)$- difference set in $G$ if and only if \label{DS equ}\[|\omega_{\chi}(D)|^2=k-\lambda\] for any $\chi\in \tup{Irr}(G)^*$.
\item[(2)] $D$ is a regular central $(v,k,\lambda,\mu)$- partial difference set in $G$ if and only if $\omega_{\chi}(D)$ is a real number and \[\omega_{\chi}(D)^2-(\lambda-\mu)\omega_{\chi}(D)-(k-\mu)=0\] for any $\chi\in \tup{Irr}(G)^*$.
\end{itemize}
\end{thm}

\begin{proof}
Since $D\in Z(\bC [G])$ by Lemma \ref{lem_Z(CG)}, for any non-trivial irreducible representation $\Phi$ of $G$, we deduce that $\Phi(D)=cI$ for some $c\in \bC$ by Schur's Lemma. Let $\chi\in \tup{Irr}(G)^*$ be the corresponding character of $\Phi$. Then
\[
\chi(D)=\tup {Tr}(\Phi(D))=c\chi(1),\] where $\tup {Tr}(\Phi(D))$ denotes the trace of the matrix $\Phi(D)$.
 So we have $\Phi(D)=\omega_{\chi}(D)I$ by the definition of $\omega_{\chi}$ in Definition \ref{wx}.  Choose a suitable basis so that the representation is unitary. Then $\Phi(D^{-1})$ is the conjugate transpose of $\Phi(D)$ and $\omega_{\chi}(D^{-1})$ is the complex conjugate of $\omega_{\chi}(D)$. Thus the results follow from Theorem $\ref{thm_pdsds}$.
\end{proof}

\subsection{ Linking systems of difference sets}
 \begin{definition}\cite{linking system2}
Let $G$ be a group of order $v$, and let  $l\geq 2$ be an integer. Suppose $\mathcal{R}=\{D_1,D_2,\cdots, D_l\}$ is a collection of size $l$ of $(v,k,\lambda, n)$-difference sets in $G$, where $n=k-\lambda$. Then $\mathcal{R}$ is  a \textbf{reduced $(v,k,\lambda,n;l)$-linking system} of difference sets in $G$ of size $l$ if there are integers $\mu$ and $\eta$ such that for all distinct $i,j$ there is some $(v,k,\lambda,n)$-difference set $D_{(i,j)}$ in $G$ satisfying
  \begin{equation}\label{linking system}
     D_i D_j^{-1}=(\mu-\eta)D_{(i,j)}+\eta G.
  \end{equation}

\end{definition}
Note that the difference set $D_{(i,j)}$ in $\eqref{linking system}$ may be not in the reduced linking system $\mathcal{R}$.
\begin{remark}\label{details ab Ls}
  A reduced $(v,k,\lambda,n;l)$-linking system of difference sets in a group $G$ with respect to integers $\mu$ and $\eta$ is equivalent to a $(v,k,\lambda,n;l+1)$-linking system of difference sets in $G$ with respect to $\mu$ and $\eta$ \cite{linking system1,linking system2}.
Thus for simplicity we only consider the reduced linking systems of difference sets in $G$ in the following.
\end{remark}

\begin{lemma}\cite[Lemma 2.2]{linking system2}\label{lem_lm(mn,nu)}
 Suppose $\mathcal{R}=\{D_1,D_2,\cdots, D_l\}$ is a reduced $(v,k,\lambda,n;l)$-linking system of difference sets in a group $G$ with respect to integers $\mu$ and $\eta$. Then
  \begin{equation}\label{exp_munv}
  \eta=\frac{k(k\pm \sqrt{n})}{v} \quad \tup{ and }\quad \mu=\eta\mp\sqrt{n}.
   \end{equation}
\end{lemma}

From Lemma \ref{lem_lm(mn,nu)} and Theorem \ref{thm_2groups}, we can obtain that if there exists a linking system $\mathcal{R}$ of difference sets in a $2$-group of order $2^{2m}$, then $\eta=2^{m-2}(2^m-1)$ and $\mu-\eta=-2^{m-1}$.

\begin{thm}\label{LSequva}
  Suppose $D_1$ and $D_2$  are two distinct central  difference sets in a non-abelian $2$-group $G$ of order $2^{2m}$. Suppose $D$ is also a central difference set in $G$.
  Let $\omega_\chi$ be as in Definition \ref{def_omega} for each $\chi\in\tup{Irr}(G)^*$.
  Then
  \[D_1D_2^{(-1)}=-2^{m-1}D+2^{m-2}(2^m-1)G
  \]if and only if
  \begin{equation}\label{char of LS}
    \omega_{\chi}(D_1)\overline{\omega_{\chi}(D_2)}=-2^{m-1}\omega_\chi(D)
  \end{equation}
  holds for any  $\chi\in \tup{Irr}(G)^*$.
  \end{thm}
The key point of this proof is similar to that of Theorem \ref{thm_pdsds}, we no longer prove here. It is worth noting that  Theorem \ref{LSequva} only holds for central difference sets.

\subsection{The Suzuki $p$-groups of type A}\quad

Suppose that $p$ is a prime and  $m$ is a positive integer. Let $\bF_{p^m}$ be the finite field with $p^m$ elements. Suppose $\theta$ is a field automorphism of $\bF_{p^m}$ such that $\theta(a)=a^{p^l}$ for each $a\in\bF_{p^m}$ and for some positive integer $l$.
Set
\[
e=\gcd(l,m),\, f=m/e.
\]
Then $o(\theta)=f$ and $\bF_{\theta}:=\{x\in\bF_{p^m}:\theta(x)=x\}$ is the fixed subfield of $\theta$, and $\bF_{\theta}=\bF_{p^e}$.
\begin{definition}\cite{Suzuki}
 The \textbf{Suzuki $p$-group $A_p(m,\theta)$} of type $A$ defined by a field automorphism $\theta$ of $\bF_{p^m}$ is the set
\[
\bF_{p^m}\times \bF_{p^m}=\{(a,b):a,b\in \bF_{p^m}\}
\]
with multiplication
\[
(a,b)\cdot(c,d)=(a+c,b+d+a\theta(c)),
\]
for any $(a,b),(c,d)\in \bF_{p^m}\times \bF_{p^m}$.
\end{definition}
For simplicity, we set $G:=A_p(m,\theta)$.
Note that if $\theta$ is  trivial,  $G$ is  abelian. We only consider the non-abelian cases here, thus $f>1$.
\begin{definition}\cite[Definition 2.1]{character of Suzuki}\label{f_a}
 For any non-trivial field automorphism $\theta$ of $\bF_{p^m}$ and  $a\in \bF_{p^m}^{*}$, define
\begin{equation}\label{eq_f_a}
f_{a,\theta}:\bF_{p^m}\rightarrow \bF_{p^m}, \; x\mapsto a\theta(x)-x\theta(a).
\end{equation}
\end{definition}
Then $f_{a,\theta}$ is $\bF_\theta$-linear  and $\textup{Im}(f_{a,\theta})$ is an $\bF_{\theta}$-hyperplane of $\bF_{p^m}$.
 Hence for any $j_a\in \bF_{p^m}\setminus\textup{Im}(f_{a,\theta})$, we have
\begin{equation}\label{Eqn_fpm}
\bF_{p^m}=\textup{Im}(f_{a,\theta})\oplus j_a\bF_{p^e}.
\end{equation}

 Let $Z(G)$  and $G'$ be the center and derived subgroup of $G$, respectively. We can read off from \cite{character of Suzuki} Table 3 that  $Z(G)=\{(0,x):x\in\bF_{p^m}\}$ and
 \[
 G'=
 \begin{cases}
 \{(0,x):x\in\tup{Im}(f_{1,\theta})\},& \tup{ if } f=2,\\
 \{(0,x):x\in\bF_{p^m}\}, & \tup{ if } f>2.
\end{cases}
 \]

\begin{lemma} \cite[Lemma 2.8]{character of Suzuki}\label{cc}
Let $G=A_p(m,\theta)$.
Then the class number of $G$ is equal to $p^{m+e}+p^m-p^e$ and all the conjugacy classes of $G$ are listed as follows:
\begin{itemize}
	\item[(1)] The set of conjugacy classes of size $1$ is $\{C_b:b\in\bF_{p^m}\}$, where $C_b=\{(0,b)\}$.
	\item[(2)] The set of conjugacy classes of size $p^{m-e}$ is  $\{C_{a,x}:a\in\bF_{p^m}^*,x\in\bF_{p^e}\}$, where \\$C_{a,x}=\{(a,j_ax+y):y\in \textup{Im}(f_{a,\theta})\}$ and $j_a\in \bF_{p^m}\setminus \tup{Im}(f_{a,\theta})$.
\end{itemize}
Thus $G=\cup_{b\in\bF_{p^m}}C_b\cup_{a\in\bF_{p^m}^*,x\in\bF_{p^e}}C_{a,x}$ is the partition of $G$ into conjugacy classes.
\end{lemma}

Write $\xi_p=\exp{(\frac{2\pi\sqrt{-1}}{p})}$  and let $\textup{Tr}_m$ and $\textup{Tr}_e$ denote the absolute trace function from $\bF_{p^m}$ to $\bF_p$ and $\bF_{p^e}$ to $\bF_p$, respectively. Define
  \begin{equation}\label{psi_v}
    \psi_v(x)=\xi_p^{\textup{Tr}_m(vx)}
  \end{equation}
for each $v\in \bF_{p^m}$ and each $x\in \bF_{p^m}$, and define
    \begin{equation}\label{phi_w}
    \phi_w(y)=\xi_p^{\textup{Tr}_e(wy)}
  \end{equation}
   for each $y\in \bF_{p^e}$ and each $w\in \bF_{p^e}$.

Denote  the set of linear characters of $G$ and the set of non-linear irreducible characters of $G$ of degree $r$ by $\tup{Lin(G)}$ and $\tup{Irr}_{(r)}(G)$, respectively.
 Now we list the character tables of $A_p(m,\theta)$ in  the following two parts.

\subsubsection{The case where $f$ is even}\label{keven}\quad

 When $f=2$, we have $l=e=\frac{m}{2}$ and $\theta(a)=a^{p^e}$ for  each $a\in \bF_{p^m}$.
  From \cite[Lemma 2.2]{character of Suzuki},  we obtain that  $\textup{Im}(f_{a,\theta})=\tup{Ker}(\tup{Tr}_{\bF_{p^m}/\bF_{p^e}})$ for any $a\in\bF_{p^m}^*$, where $\tup{Tr}_{\bF_{p^m}/\bF_{p^e}}$ denotes the relative trace function from $\bF_{p^m}$ to $\bF_{p^e}$.
  Choose $j\in\bF_{p^m}$ such that $j+\theta(j)=1$. Then $j\notin \tup{Ker}(\tup{Tr}_{\bF_{p^m}/\bF_{p^e}})$, and $\bF_{p^{m}}=\tup{Ker}(\tup{Tr}_{\bF_{p^m}/\bF_{P^e}})\oplus j\bF_{p^e}$.
Thus Table \ref{table_A(p,2)} follows directly from Table 6 in \cite{character of Suzuki}.
\begin{table}[H]
\caption{The character table of $G:=A_p(m,\theta)$ with $f=2$} \label{table_A(p,2)}
\begin{tabular}{l|l|l}
\hline
&$(0,b)\in C_b,b\in\bF_{p^m} $&$(a,jx+y)\in C_{a,x},a\in\bF_{p^{m}}^*,x\in\bF_{p^e}$\\
\hline
$\chi_1^{(v,w)}\in \tup{Lin}(G), v\in\bF_{p^m},w\in\bF_{p^e}$&   $\psi_w(b)$&$\psi_v(a)\phi_w(x-a\theta(a))$\\
\hline
$\chi_{p^e}^v\in\tup{Irr}_{(p^e)}(G), v\in\bF_{p^m}\setminus\bF_{p^e}$&   $ p^e\psi_v(b)$&\quad\quad0\\
\hline
\end{tabular}
\end{table}
When $f>2$ is even, we take the notations listed in \cite{character of Suzuki}. Let $T$ be a set of coset representatives for $\bF_p^*$ in $\bF_{p^m}^*$ and $\gamma$ be a primitive element of $\bF_{p^m}^*$. Set $J_1'=T\setminus (\langle \gamma^{p^e+1}\rangle \cap T)$.  Then Table \ref{table_A(p,even)} follows from Table 8 in \cite{character of Suzuki}.
In this table, we omit the expressions of  irreducible characters of degree $p^{\frac{m-2e}{2}}$ acting  on conjugacy classes of $G$, since we will not use them in this paper.
\begin{table}[H]
\caption{The character table of $G:=A_p(m,\theta)$ with $f>2$ even} \label{table_A(p,even)}
\begin{tabular}{c|c|c}
\hline
&$(0,b)\in C_b, b\in\bF_{p^m}$&$(a,j_ax+y)\in C_{a,x},a\in\bF_{p^m}^*, x\in\bF_{p^e}$\\
\hline
$\chi_1^v\in \tup{Lin}(G),v\in\bF_{p^m}$&$1$&$\psi_v(a)$\\
\hline
$\chi_{p^{m/2}}^{(v,s)}\in\tup{Irr}_{(p^{m/2})}(G), v\in J_1', 1\leq s\leq p-1$&$p^{m/2}\psi_{sv}(b)$&$0$\\

\hline
*&*&*\\
\hline
\end{tabular}
\end{table}
\subsubsection{The case where $f>1$ is odd}\label{k odd}\quad

Suppose $p\nmid f$, then $\tup{Tr}_{\bF_{p^m}/\bF_{p^e}}((a\theta(a))^{-1}a\theta(a))=f\neq0$.
So, from \cite[Lemma 2.2]{character of Suzuki}, we can take $j_a=a\theta(a)$.

When $p=2$, let $U=\{x\in\bF_{2^e}:\tup{Tr}_e(x)=0\}$. Fix $u_0\in\bF_{2^e}$ such that $\tup{Tr}_e(u_0)=1$. Then for each $u\in\bF_{2^e}$, we have that $u=\delta u_0+u_1$ for some $\delta\in\bF_2$ and $ u_1\in U$.
For each $v\in \bF_{2^m}^*$, chose $a_v$ such that  $j_{a_v}=a_v \theta(a_v)=v^{-1}$.
 We define an auxiliary function
\begin{equation}\label{kk}
\kappa:\bF_2\rightarrow \{0,1\}\subset \bZ ,\,\, \overline{0}\mapsto 0,\overline{1}\mapsto 1.
\end{equation}
Suppose
 $Q(x)$ is a function from $\bF_{2^e}$ to $\bF_{2}$ given by
\begin{equation}\label{eq_Qx}
 Q(x)=
\begin{cases}
\sum_{s=0}^{\frac{e-1}{2}}\tup{Tr}_e(x^{2^s+1}),&\tup{ if }e\tup{ is odd},\\
\sum_{s=0}^{\frac{e}{2}-1}\tup{Tr}_e(x^{2^s+1})+\tup{Tr}_e(cx^{2^{e/2+1})},&\tup{ if }e\tup{ is even},
\end{cases}
\end{equation}

where $c\in\bF_{2^e}$ such that $c+c^{2^{e/2}}=1$. Write $i=\sqrt{-1}$. Then Table \ref{tableA(2,odd)} follows from Table 7 in \cite{character of Suzuki}.
\begin{table}[H]\label{table1}
\caption{The character table of $G:=A_2(m,\theta)$ with $f>1$ odd} \label{tableA(2,odd)}
\begin{tabular}{c|c|c|c}
\hline
&$(0,b)\in C_b$&\multicolumn{2}{c}{$(a,j_ax+y)\in C_{a,x}$}\\
&$b\in\bF_{2^m}$&\multicolumn{2}{c}{$a\in\bF_{2^m}^*, x\in\bF_{2^e}$}\\
\hline
$\chi_1^v\in \tup{Lin}(G),v\in\bF_{p^m}$&$1$&\multicolumn{2}{c}{$\psi_v(a)$}\\
\hline
$\chi_{2^{\frac{m-e}{2}}}^{(v,w,\epsilon)}\in \tup{Irr}_{(2^{\frac{m-e}{2}})}(G)$&\multirow{2}{*}{$ 2^{\frac{m-e}{2}}\psi_v(b)$}&
$\tup{ if }a=a_vu=a_v(\delta u_0+u_1)\in a_v\bF_{2^e}^*$&$\tup{ if }a\not\in a_v\bF_{2^e}^*$\\

$v\in\bF_{2^m}^*,w\in\bF_{2^e}/\bF_2, \epsilon=\pm 1$&&$2^{\frac{m-e}{2}}(\epsilon i)^{\kappa(\delta)}(-1)^{Q(u_1)+\tup{Tr}_{e}(\delta u_0u_1)}\phi_{w}(u_1)\phi_{u^2}(x)$&$0$\\
\hline
\end{tabular}
\end{table}
We use the notation $\square_{p^m}$ ( resp. $\square_{p^e}$) for the set of non-zero  squares of $\bF_{p^m}$ ( resp. $\bF_{p^e}$ )and  $\blacksquare_{p^m}$ ( resp. $\blacksquare_{p^e}$) for the set of non-squares of $\bF_{p^m}$ ( resp. $\bF_{p^e}$).
  When $p$ is odd, fix $x_0\in\blacksquare_{p^e}$. Set
\begin{equation}\label{eq_xv}
x_v=
\begin{cases}
  1, &\tup{ if } v\in\square_{p^m},\\
  x_0, &\tup{ if } v\in\blacksquare_{p^m}.
\end{cases}
\end{equation}
For each $v\in\bF_{p^m}^*$, choose $a_v$ such that  $vj_{a_v}=va_v\theta(a_v)=x_v$.
Let $T$ be a set of coset representatives for $\bF_p^*$ in $\bF_{p^m}^*$. Then Table  \ref{tableA(p,odd)} follows from Table 7 in \cite{character of Suzuki}.
\begin{table}[H]\label{table1}
\caption{The character table of $G:=A_p(m,\theta)$ with $p>2$ and $f>1$ odd} \label{tableA(p,odd)}
\begin{tabular}{c|c|c|c}
\hline
&$(0,b)\in C_b$&\multicolumn{2}{c}{$(a,j_ax+y)\in C_{a,x}$}\\
&$b\in\bF_{p^m}$&\multicolumn{2}{c}{$a\in\bF_{p^m}^*, x\in\bF_{p^e}$}\\
\hline
$\chi_1^v\in \tup{Lin}(G),v\in\bF_{p^m}$&$1$&\multicolumn{2}{c}{$\psi_v(a)$}\\
\hline
$\chi_{p^{\frac{m-e}{2}}}^{(v,w,s)}\in \tup{Irr}_{(p^{\frac{m-e}{2}})}(G)$&\multirow{2}{*}{$ p^{\frac{m-e}{2}}\psi_{sv}(b)$}&
$\tup{ if }a=a_vu\in a_v\bF_{p^e}^*$&$\tup{ if }a\not\in a_v\bF_{p^e}^*$\\

$v\in T,w\in\bF_{p^e}, 1\leq s\leq p-1$&&$p^{\frac{m-e}{2}}\xi_p^{-\frac{1}{2}sf\tup{Tr}_e(x_vu^2)}\phi_{w}(u)\phi_{sf x_vu^2}(x)$&$0$\\
\hline
\end{tabular}
\end{table}

\section{Non-existence results in $A_p(m,\theta)$ with $f$  even}\label{nonexistence}
Set $G:=A_p(m,\theta)$. Recall that $f=o(\theta),e=m/f$.
In this section, we prove the non-existence of non-trivial central difference sets in $G$ when $p=2$. Similarly, when $p>2$, we
give some non-existence results for Latin square type central partial difference sets in $G$.

\begin{lemma}\label{orthogonal}
  Let $H$ be an abelian additive group, and let $K$ be any subgroup of $H$. Then for each  $\chi\in \tup{Irr}(G)^*$, we have
 \[ \chi(K)=\begin{cases}
          |K|,&\tup{ if } K\subseteq \tup{Ker}\chi,\\
          0, &\tup{ if } K\nsubseteq \tup{Ker}\chi.
          \end{cases}\]
 \begin{proof}
 If $K\nsubseteq \tup{Ker}\chi$,  there exists  $g\in K$ such that $\chi(g)\neq 1$. Then $\chi(K)=\chi(g+K)=\chi(g)\chi(K)$. It deduces that $\chi(K)=0$.
 \end{proof}
\end{lemma}
\begin{thm} Let $G=A_2(m,\theta)$ and suppose $f$ is even. Then there exists no non-trivial central difference sets in  $G$.
\end{thm}
\begin{proof}
We assume that $D$ is a non-trivial central difference set in $G$.  According to Theorem \ref{thm_2groups}, $D$ has  parameters \[(2^{2m},2^{m-1}(2^m-1),2^{m-1}(2^{m-1}-1),2^{2(m-1)}).\] Moreover, $|\omega_{\chi}(D)|= 2^{m-1}$ for any non-trivial irreducible character $\chi\in\tup{Irr}(G)^*$ by  Theorem $\ref{char theorem}$. Since $D$ is central, it can be written as \[D=\cup_{b\in B}C_b\cup_{a\in \bF_{2^m}^*,x\in \Gamma(a)}C_{a,x}\] for some $B\subseteq\bF_{2^m}$ and $\Gamma(a)\subseteq \bF_{2^e}$, $a\in\bF_{2^m}^*$, where $C_b$ and $C_{a,x}$ are the conjugacy classes defined in Lemma \ref{cc}.
Since the character tables of $G$ with $f=2$ and $f>2$ even are different, we will consider these two cases separately.

In the case $f=2$, for any non-linear irreducible character $\chi_{2^e}^v$, $v\in \bF_{2^{m}}\setminus \bF_{2^e}$ listed in Table \ref{table_A(p,2)} we have
\begin{equation}\label{psi_v(B)}
 \omega_{\chi_{2^e}^v}(D)=\psi_v(B)=\pm2^{m-1}.
\end{equation}
 It is deduced from Lemma \ref{orthogonal} that $\psi_v(\bF_{2^m})=0$ for each $v\in\bF_{2^m}^*$. Noting that $\psi_v(a)=\pm 1$ for any $a\in \bF_{p^m}$, we have $|\psi_v(B)|\leq 2^{m-1}$ for any subset $B$ of $\bF_{2^m}$. Besides,  the equality holds if and only if
 \[
B=\tup{Ker}(\psi_v) \tup{ or } B=\bF_{2^m}\backslash\tup{Ker}(\psi_v).\]
Thus  \eqref{psi_v(B)} implies that $B=\tup{Ker}(\psi_v)$ or $\bF_{2^m}\backslash\tup{Ker}(\psi_v)$ for each  $v\in \bF_{2^{m}}\setminus \bF_{2^e}$. Then we get a contradiction.

In the case where $f>2$ is even,  for any irreducible character $\chi_{2^{m/2}}^{(v,1)}$, $v\in J_1'$ listed in Table \ref{table_A(p,even)}, we have
\[
\omega_{\chi_{2^{m/2}}^{(v,1)}}(D)=\psi_v(B)=\pm2^{m-1}.
\]
Similar to the proof of  $f=2$, we get a contradiction by $|J_1'|\geq2$. Thus the result follows.
\end{proof}

\begin{thm}
Let $G=A_p(m,\theta)$. Suppose that $f$ is even and $p$ is odd. Then there exists no regular Latin square type central  partial difference sets in $G$ with parameters
\begin{equation}\label{eq_para1}
  (p^{2m},p^{m-1}(p^m-1), p^{2(m-1)}+p^m-3p^{m-1},p^{2(m-1)}-p^{m-1})
\end{equation}
or
\begin{equation}\label{eq_para2}
 (p^{2m},(p^{m-1}+1)(p^m-1),p^{2(m-1)}+p^m-p^{m-1}-2,p^{2(m-1)}+p^{m-1} ).
\end{equation}

\end{thm}
\begin{proof}
We first claim that for any $v\in\bF_{p^m}^*$ and for any subset $B$ of $\bF_{p^m}$, if $\psi_v(B)$ is an integer, then \[-p^{m-1}\leq\psi_v(B)\leq p^{m-1}.\]
Choose $w\in\bF_{p^m}$ such that $\tup{Tr}_m(vw)=1$, where we recall that $\textup{Tr}_m$ denotes the absolute trace function from $\bF_{p^m}$ to $\bF_p$.  Then we have
\[\bF_{p^m}=\cup_{s\in\bF_p}(sw+\tup{Ker}(\psi_v)).\]
Since $\psi_v(B)$ is an integer, we can write  $B=\cup_{s\in\bF_p}(sw+W_s)$, where $W_s\subseteq \tup{Ker}(\psi_v)$ has equal size for all $s\in\bF_p^*$. We assume that $a=|W_0|$ and $b=|W_s|$ for $s\in\bF_p^*$.
Then $\psi_v(B)=a-b$. Since $0\leq a,b\leq p^{m-1}$, we have $-p^{m-1}\leq\psi_v(B)\leq p^{m-1}$. Besides, the equality holds if and only if $a=0,b=p^{m-1}$ or $a=p^{m-1}, b=0$. That is,
\[
B=\bF_{p^m}\setminus\tup{Ker}(\psi_v)\tup{ or }B=\tup{Ker}(\psi_v).\]

Suppose
\[D=\cup_{b\in B}C_b\cup_{a\in \bF_{p^m}^*,x\in \Gamma(a)}C_{a,x},\, \tup{ where }B\subseteq\bF_{p^m}, \Gamma(a)\subseteq\bF_{p^e}, a\in\bF_{p^m}^*.\]
When $f=2$, if $D$ is a central partial difference set with  parameters in \eqref{eq_para1}, form Theorem \ref{char theorem} and Table \ref{table_A(p,2)}, we have
\[\omega_{\chi_{p^e}^v}(D)=\psi_v(B)\in \{-p^{m-1},p^{m-1}(p-1)\},\, \tup{ for any }v\in\bF_{p^m}\setminus\bF_{p^e}.\]
Since $p^{m-1}< p^{m-1}(p-1)$, we obtain that $\psi_v(B)=p^{m-1}$ for any  $v\in\bF_{p^m}\setminus\bF_{p^e}$, this actually can not come true.
Similarly, if $D$ is a central partial difference set with  parameters in \eqref{eq_para2},  we have
\[\omega_{\chi_{p^e}^v}(D)=\psi_v(B)\in \{-p^{m-1}-1,p^m-(p^{m-1}+1)\},\,\tup{ for any }v\in\bF_{p^m}\setminus\bF_{p^e}.\]
Noting that \[-p^{m-1}-1<-p^{m-1}\tup{ and }p^{m-1}< p^m-(p^{m-1}+1),\] we get a contradiction. Thus there exists no central partial difference sets in $G$ with those two kinds of parameters listed in \eqref{eq_para1} and \eqref{eq_para2}.
When  $f>2$ is even, the proof is similar to $f=2$ and we do not give the details here.
\end{proof}
\begin{remark}
 In  \cite{Gow}, Gow and Quinlan have a similar result that there exist no non-trivial central difference sets  in the group of Suzuki type with order $2^{2s}$ under the condition that $s$ is even.
 Note that the Suzuki $2$-group of type $A$  is usually not a  group of Suzuki type defined in \cite{Gow}, unless $o(\theta)=m$.
 Thus the research of central difference sets in $A_2(m,\theta)$ is still meaningful.
\end{remark}
\section{New central difference sets in $A_p(m,\theta)$ with $f>1$  odd }\label{constructions}
 In this section,  we give the constructions of  central difference sets in $ G:=A_p(m,\theta)$ when $p=2$ and  Latin square type central partial difference sets in $G$ when $p$ is odd. We will use the notations mentioned in Section \ref{k odd}.
\subsection{Central difference sets in $G:=A_2(m,\theta)$}\quad

Recall that  \[U=\{x\in \bF_{2^e}: \tup{Tr}_e(x)=0\}\tup{ and  }\tup{Tr}_e(u_0)=1.\] For each $u\in\bF_{2^e}$, we write $u=\delta u_0+u_1$, where $\delta\in\bF_2, u_1\in U$. For each $t\in \bF_{2^m}^*$, take $a_t$ such that $t a_t\theta(a_t)=1$.
\begin{thm}\label{DS1}
 Let $G=A_2(m,\theta)$ with $f>1$ odd.
For each $t\in \bF_{2^m}^*$,  take $B_t=\tup{Ker}(\psi_t)$ or $\bF_{2^m}\setminus \tup{Ker}(\psi_t)$, where $\psi_t$ is as in \eqref{psi_v}. Choose $z\in\bF_{2^e}^*$ and write $\sqrt{z}=z^{2^{e-1}}$. Let $\phi_z$ be as in \eqref{phi_w} and
 set
 \[
 \Gamma(a)=
 \begin{cases}
  \tup{Ker}(\phi_z) \tup{ or } \bF_{2^e}\setminus \tup{Ker}(\phi_z), &\tup{ if }a\in\bF_{2^m}^*\setminus\{a_t\sqrt{z}\},\\
   \varnothing,  &\tup{ if }a=a_t\sqrt{z}.
 \end{cases}
 \]
  Suppose
 \[
D_{(t,z)}=\cup_{b\in B_t}C_b\cup_{a\in \bF_{2^m}^*,x\in \Gamma(a)}C_{a,x},
\]
where $C_b$ and $C_{a,x}$ are the conjugacy classes defined in Lemma \ref{cc}.
  Then $D_{(t,z)}$ is a central difference  set in $G$ with parameters $(2^{2m},2^{m-1}(2^m-1),2^{m-1}(2^{m-1}-1))$.
\end{thm}
\begin{proof}
According to Theorem \ref{char theorem}, we only need to show that $|\omega_\chi(D_{(t,z)})|= 2^{m-1}$ for each $\chi\in\tup{Irr}(G)^*$.

We first consider the linear characters of $G$. For any  $\chi_1^v$, $v\in \bF_{2^m}^*$ listed in Table $\ref{tableA(2,odd)}$, from Lemma $\ref{cc}$ we have
\begin{align}\label{eq_lin}
   \omega_{\chi_1^v}(D_{(t,z)})=& |B_t|+2^{m-e}\sum_{a\in \bF_{2^m}^*\setminus \{a_t\sqrt{z}\}}\psi_v(a)|\Gamma(a)|\notag \\
  = & |B_t|+2^{m-e}\cdot2^{e-1}\sum_{a\in \bF_{2^m}^*\setminus \{a_t\sqrt{z}\}}\psi_v(a)\notag\\
  =&-2^{m-1}\psi_v(a_t\sqrt{z})=\pm 2^{m-1}.
\end{align}

Then we consider the non-linear irreducible characters of $G$.
For any $\chi_{2^{\frac{m-e}{2}}}^{(v,w, \epsilon)}, v\in\bF_{2^m}^*,w\in\bF_{2^e}/\bF_2,\epsilon=\pm 1$ listed in Table $\ref{tableA(2,odd)}$, we have
\begin{equation*}
   \omega_{\chi_{2^{\frac{m-e}{2}}}^{(v,w,\epsilon)}}(D_{(t,z)}) = \psi_v(B_t)+2^{m-e}\sum_{u=\delta u_0+u_1\in \bF_{2^e}^*}(\epsilon i)^{\kappa(\delta)}(-1)^{Q(u_1)+\tup{Tr}_e(\delta u_0u_1)}\phi_w(u_1)\phi_{u^2}(\Gamma(a_vu)).
\end{equation*}
 We deduce from Lemma $\ref{orthogonal}$ that if $u^2\neq z$ then $\phi_{u^2}(\Gamma(a_vu))=0$.
Suppose that \[\sqrt{z}=\delta u_0+u_1, \delta\in\bF_2, u_1\in U.\]
Then we have \[
 \omega_{\chi_{2^{\frac{m-e}{2}}}^{(v,w,\epsilon)}}(D_{(t,z)}) =\psi_v(B_t)+2^{m-e}(\epsilon i)^{\kappa(\delta)}(-1)^{Q(u_1)+\tup{Tr}_e(\delta u_0u_1)}\phi_w(u_1)\phi_{z}(\Gamma(a_v\sqrt{z})).
\]
Note that\[\psi_v(B_t)=
\begin{cases}
\pm2^{m-1}, &\tup {if } v=t,\\
0, &\tup {if }  v\neq t,
\end{cases}
\]
and
\[\phi_z(\Gamma(a_v\sqrt{z}))=
\begin{cases}
0, &\tup {if }  v= t,\\
\pm2^{e-1}, &\tup {if } v\neq t.
\end{cases}
\]
So we have
\begin{equation}\label{eq_nonli}
 \omega_{\chi_{2^{\frac{m-e}{2}}}^{(v,w,\epsilon)}}(D_{(t,z)})=\begin{cases}
\psi_v(B_t)=\pm2^{m-1}, &\tup {if } v=t,\\
\pm2^{m-1}(\epsilon i)^{\kappa(\delta)}(-1)^{Q(u_1)+\tup{Tr}_e(\delta u_0u_1)}\phi_w(u_1), &\tup {if }  v\neq t.\end{cases}
\end{equation}
Therefore, $|\omega_{\chi_{2^{\frac{m-e}{2}}}}^{(v,w, \epsilon)}(D_{(t,z)})|= 2^{m-1}$  for any non-linear irreducible character $\chi_{2^{\frac{m-e}{2}}}^{(v,w, \epsilon)}\in \tup{Irr}_{(2^{\frac{m-e}{2}})}(G)$.

\end{proof}
\begin{thm}\label{DS2}
  Let $G=A_2(m,\theta)$ with $f>1$ odd.
  For each $z\in\bF_{2^e}^*$, let $\phi_z$ be as in \eqref{phi_w}. Take
  \[D_{z}=\cup_{a\in\bF_{2^m}^*, x\in \Gamma(a)}C_{a,x},\] where
 $\Gamma(a)=\tup{Ker}(\phi_z)$  or $\bF_{2^e}\setminus\tup{Ker}(\phi_z)$  for any $a\in\bF_{2^m}^*$.

  Then $D_z$ is a central difference set in $G$ with parameters $(2^{2m},2^{m-1}(2^m-1),2^{m-1}(2^{m-1}-1),2^{2(m-1)})$.
\end{thm}
\begin{proof}
Similar to the proof of Theorem $\ref{DS1}$, we only need to show that $|\omega_\chi(D_{(t,z)})|= 2^{m-1}$ for each $\chi\in\tup{Irr}(G)^*$.
 For any  $\chi_1^v$, $v\in \bF_{2^m}^*$, listed in Table$ \ref{tableA(2,odd)}$, from Lemma $\ref{cc}$
we have
\begin{equation}\label{linD1}
  \omega_{\chi_1^v}(D_z)=2^{m-n}\cdot2^{e-1}\sum_{a\in \bF_{p^m}^*}\psi_v(a)=2^{m-1}\cdot(-1)=-2^{m-1}.
\end{equation}
 For any  $\chi_{2^{\frac{m-e}{2}}}^{(v,w, \epsilon)} \in \tup{Irr}_{(2^{\frac{m-e}{2}})}(G)$ listed in Table $\ref{tableA(2,odd)}$, we have
\begin{align}\label{nonlinD1}
 \omega_{\chi_{2^{\frac{m-e}{2}}}^{(v,w,\epsilon)}}(D_z)&=2^{m-e}\sum_{u=\delta u_0+u_1\in \bF_{2^e}^*}(\epsilon i)^{\kappa(\delta)}(-1)^{Q(u_1)+\tup{Tr}_e(\delta u_0u_1)}\phi_w(u_1)\phi_{u^2}(\Gamma(a_vu))
\end{align}
Write $\sqrt{z}=z^{2^{e-1}}$ and suppose $\sqrt{z}=\delta u_0+u_1, \,\delta\in\bF_2, u_1\in U$. Then
\begin{align}\label{nonlinD}
  \eqref{nonlinD1}&=2^{m-e}(\epsilon i)^{\kappa(\delta)}(-1)^{Q(u_1)+\tup{Tr}_e(\delta u_0u_1)}\phi_w(u_1)\phi_{z}(\Gamma(a_vz))\nonumber\\
&=\pm2^{m-1}(\epsilon i)^{\kappa(\delta)}(-1)^{Q(u_1)+\tup{Tr}_e(\delta u_0u_1)}\phi_w(u_1).
\end{align}
Therefore, $|\omega_{\chi_{2^{\frac{m-e}{2}}}^{(v,w,\epsilon)}}(D_z)|=2^{m-1}$.
\end{proof}

\begin{prop}
 The constructions of central difference sets given in Theorem \ref{DS1} and Theorem \ref{DS2} are both covered by Dillon's construction listed in Lemma \ref{Lem_DIllon}.
 \begin{proof}
Consider the difference set $D_{(t,z)}=\cup_{b\in B_t}C_b\cup_{a\in \bF_{2^m}^*,x\in \Gamma(a)}C_{a,x}$ given in Theorem \ref{DS1}.
 For any $a\in\bF_{2^m}^*$, we set
 \begin{align*}
 X_a&=\{(0,x):x\in j_a\tup{Ker}(\phi_z)\oplus \tup{Im}(f_{a,\theta})\}.
 \end{align*}
Note that $\{ j_a\tup{Ker}(\phi_z)\oplus \tup{Im}(f_{a,\theta}):a\in\bF_{2^m}^*\}$ are all the hyperplanes of $\bF_{2^m}$ and
\[
X_{a_t\sqrt{z}}=\{(0,x):x\in j_{a_t}z\tup{Ker}(\phi_z)\oplus\tup{Im}(f_{a_t,\theta})\}=\{(0,x):x\in\tup{ker}(\psi_t)\}.
\]
 Take $c\in\bF_{2^m}\setminus \tup{Ker}(\psi_t)$ and $d\in\bF_{2^n}\setminus \tup{Ker}(\phi_z)$, we have
 \begin{align*}
 \cup_{b\in B_t}C_b=\begin{cases}X_{a_t\sqrt{z}},&\tup{ if }B_t=\tup{Ker}(\psi_t),\\(0,c)X_{a_t\sqrt{z}},&\tup{ if }B_t=\bF_{p^m}\setminus\tup{Ker}(\psi_t),\end{cases}
 \end{align*}
 and for any $a\in\bF_{2^m}^*$,
 \begin{align*}
 \cup_{x\in \Gamma(a)}C_{a,x}=\begin{cases}(a,0)X_a,&\tup{ if }\Gamma(a)= \tup{Ker}(\phi_z),\\(a,j_ad)X_a,&\tup{ if }\Gamma(a)=\bF_{2^e}\setminus \tup{Ker}(\phi_z).\end{cases}
 \end{align*}
Therefore, the central difference sets $D_{(t,z)}$ given in Theorem \ref{DS1} is a Dillon's construction. Mimic the method above, we obtain that the construction of central difference sets given in Theorem \ref{DS2} is also a Dillon's construction.
 \end{proof}
\end{prop}

 Now we give two examples of central difference sets in $G:=A_2(m,\theta)$ with $m=f$ in Theorem \ref{DS1} and Theorem \ref{DS2}, respectively.
\begin{example}\label{eg1}
 For each $t\in \bF_{2^m}^*$, let $D_t=\cup_{b\in B_t}C_b\cup_{a\in J_{t}^0}C_{a,0}\cup_{a\in J_{t}^1}C_{a,1}$,
where $B_t=\tup{Ker}(\psi_t)$ or $\bF_{2^m}\setminus \tup{Ker}(\psi_t)$ , $J_{t}^0\cup J_t^1=\bF_{2^m}^*\setminus \{a_t\}$ and $J_{t}^0\cap J_t^1=\varnothing$. Then $D_t$ is a non-trivial central difference set in $G$ by Theorem \ref{DS1}.
\end{example}
\begin{example}\label{eg2}
Let $D=\cup_{a\in J^0}C_{a,0}\cup_{a\in J^1}C_{a,1}$,
where  $J^0\cup J^1=\bF_{2^m}^*$ and $J^0\cap J^1=\varnothing$. Then D is a non-trivial central difference set in $G$ by Theorem \ref{DS2}.
\end{example}
\begin{corollary}
Let $G=A_2(m,\theta)$ with $m=f$ odd. Then any   non-trivial central difference set in $G$ arises in the manner described in Example $\ref{eg1}$ or Example $\ref{eg2}$ .
\end{corollary}
\begin{proof}
 In this case, we have $\bF_{2^e}=\bF_2$.
 Suppose $D=\cup_{b\in B}C_b\cup_{a\in\bF_{2^m}^*, x\in \Gamma(a)}C_{a,x}$, where $B\subseteq\bF_{2^m},\Gamma(a)\subseteq\bF_2$.
 Then  $D$ is a central difference set in $G$ if and only if
 \[\left\{
   \begin{array}{ll}
     |D|=2^{m-1}(2^m-1), \\
     |\omega_\chi(D)|=2^{m-1} \tup{ for any } \chi\in\tup{Irr}(G)^*.
   \end{array}
 \right.
\]
 That is,
\[\left\{
   \begin{array}{ll}
     |B|+2^{m-1}\sum_{a\in\bF_{2^m}^*}|\Gamma(a)|= 2^{m-1}(2^m-1),\\
    |B|+2^{m-1}\sum_{a\in\bF_{2^m}^*}|\Gamma(a)|\psi_v(a)=\pm 2^{m-1}\tup{ for all } v\in \bF_{2^m}^*,\\
     \psi_v(B)+2^{m-1}\phi_1(\Gamma(a_v))=\pm 2^{m-1}\tup{ for all } v\in \bF_{2^m}^*.
   \end{array}
 \right.
\]
From the last equality we deduce that $2^{m-1} \mid\psi_v(B)$ for all $ v\in \bF_{2^m}^*$.
Note that $|\psi_v(B)|\leq 2^{m-1}$ for any $v\in\bF_{2^m}^*$ and the equality holds if and only if
\[B=\tup{Ker}(\psi_t)\tup{or }\bF_{2^m}\setminus \tup{Ker}(\psi_t)\tup{ for some }t\in\bF_{2^m}^*.\]  Thus we have the following two cases to consider.

Case 1. $B=\varnothing$.
It follows that $\phi_1(\Gamma(a_v))=\pm 1$ for all $ v\in \bF_{2^m}^*$, where  recall that $a_v$ is the element uniquely determined by $va_v\theta(a_v)=1$. Thus $|\Gamma(a)|=1$ for all $a\in \bF_{2^m}^*$. This corresponds to Example $\ref{eg2}$.

Case 2. $B=\tup{Ker}(\psi_t)$ or $\bF_{2^m}\setminus \tup{Ker}(\psi_t)$ for some $t\in\bF_{2^m}^*$.
In this case, we have
\[
\psi_t(B)=\pm 2^{m-1}\tup{ and }\psi_v(B)=0\tup{ for any }v\neq t\in\bF_{2^m}^*.
\]
It deduces that $\phi_1(\Gamma(a_t))=0$ and $\phi_1(\Gamma(a_v))=\pm 1$ for any $v\neq t\in\bF_{2^m}^*$. Noting that
$\sum_{a\in\bF_{2^m}^*}|\Gamma(a)|= 2^m-2$, we have
\[
|\Gamma(a_t)|=0\tup{ and }|\Gamma(a_v)|=1\tup{ for any }v\neq t\in\bF_{2^m}^*.
\]
This corresponds to Example $\ref{eg1}$.

\end{proof}
\subsection{Linking systems of central difference sets in $A_2(m,\theta)$ }\label{linking systems}\quad

  In this section, we give some reduced linking systems of difference sets in $G=A_2(m,\theta)$ with parameters \[(2^{2m},2^{m-1}(2^m-1),2^{m-1}(2^{m-1}-1),2^{2(m-1)}; 2^e-1)\] with respect to difference sets constructed in Theorem \ref{DS1} and Theorem \ref{DS2}, respectively.

\begin{thm}\label{thm_lk1}
For each $t\in \bF_{2^m}^*$, $z\in \bF_{2^e}^*$, write $\sqrt{z}=z^{2^{e-1}}$ and  let
 \[
D_{(t,z)}=\cup_{b\in B_{(t,z)}}C_b\cup_{a\in \bF_{2^m}^*\setminus\{a_t\sqrt{z}\},x\in \Gamma_{(t,z)}}C_{a,x},
\]
 where
\[
B_{(t,z)}=\tup{Ker}(\psi_t)\tup{ or }\bF_{2^m}\setminus \tup{Ker}(\psi_t),\, \Gamma_{(t,z)}=\tup{Ker}(\phi_z)\tup { or }\bF_{2^e}\setminus \tup{Ker}(\phi_z).
\]
Suppose \[\mathcal{R}_t=\{ D_{(t,z)}: z\in \bF_{2^e}^*\}.\] Then $\mathcal{R}_t$ is a
  $(2^{2m},2^{m-1}(2^m-1),2^{m-1}(2^{m-1}-1),2^{2(m-1)}; 2^e-1)$-reduced linking system of difference sets  in $G=A_2(m,\theta)$.
  \end{thm}
\begin{proof}
Note that $D_{(t,z)}$ is the central difference set given in Theorem \ref{DS1}.
From Theorem \ref{LSequva},  we shall show that for any pair $D_{(t,z)},D_{(t,z')}\in \mathcal{R}_t$  with $z\neq z'$, there always exists  $D\in\mathcal{R}_t$  such that
 \begin{equation}\label{char of LS1}
    \omega_{\chi}(D_{(t,z)})\overline{\omega_{\chi}(D_{(t,z')})}=-2^{m-1}\omega_\chi({D})
  \end{equation}
  holds for any  $\chi\in \tup{Irr}(G)^*$.
Set
\[
\eta_z=
\begin{cases}
  1,&\tup{ if } B_{(t,z)}=\tup{Ker}(\psi_t),\\
-1,&\tup{ if } B_{(t,z)}=\bF_{2^m}\setminus\tup{Ker}(\psi_t),
\end{cases}\,\]
and\[
\epsilon_z=
\begin{cases}
1,&\tup{ if } \Gamma_{(t,z)}=\tup{Ker}(\phi_z),\\
-1,&\tup{ if } \Gamma_{(t,z)}=\bF_{2^e}\setminus\tup{Ker}(\phi_z).
\end{cases}
\]
Suppose $\sqrt{z}=\delta u_0+u_1, \sqrt{z'}=\delta' u_0+u_1'$, where $\delta,\delta'\in\bF_2$, $\tup{Tr}_e(u_1)=\tup{Tr}_e(u_1')=0$.
Then we have
\[
\sqrt{z+z'}=\sqrt{z}+\sqrt{z'}=(\delta +\delta')u_0+(u_1+u_1').
\]
Define  \[D_{(t,z+z')}=\cup_{b\in B_{(t,z+z')}}C_b\cup_{a\in \bF_{2^m}^*\setminus\{a_t\sqrt{z+z'}\},x\in \Gamma_{(t,z+z')}}C_{a,x},\]
where \[ B_{(t,z+z')}=\begin{cases}
\tup{Ker}(\psi_t),&\tup{ if } -\eta_z\eta_{z'}=1,\\
\bF_{2^m}\setminus \tup{Ker}(\psi_t),&\tup{ if }  -\eta_z\eta_{z'}=-1,
 \end{cases}\]
and \[\Gamma_{(t,z+z')}=
 \begin{cases}
\tup{Ker}(\phi_{ z+z'}),&\tup{ if } -\varepsilon_z\varepsilon_{z'}(-1)^{\delta'+\delta\delta'+\tup{Tr}_e(u_1u_1'+\delta' u_0u_1+\delta u_0u_1')}=1,\\
\bF_{2^e}\setminus\tup{Ker}(\phi_{ z+z'}), &\tup{ if } -\varepsilon_z\varepsilon_{z'}(-1)^{\delta'+\delta\delta'+\tup{Tr}_e(u_1u_1'+\delta' u_0u_1+\delta u_0u_1')}=-1.
 \end{cases}\]
 By the definition of $D_{(t,z+z')}$, we have
\[
\eta_{z+z'}=-\eta_z\eta_{z'},\,\varepsilon_{z+z'}=-\varepsilon_z\varepsilon_{z'}(-1)^{\delta'+\delta\delta'+\tup{Tr}_e(u_1u_1'+\delta' u_0u_1+\delta u_0u_1')}.
\]

Now we prove that $D=D_{(t,z+z')}$.
    For each $\chi_1^v\in \tup{Lin}(G)^*$,we obtain from \eqref{eq_lin} that
    \begin{align*}
      \omega_{\chi_1^v}(D_{(t,z)})\overline{\omega_{\chi_1^v}(D_{(t,z')})} &= (-2^{m-1}\psi_v(a_t\sqrt{z}))(-2^{m-1}\psi_v(a_t\sqrt{z'})) \\
       &=-2^{m-1}(-2^{m-1})\psi_v{(a_t\sqrt{z+z'})}\\
       &=-2^{m-1}\omega_{\chi_1^v}(D_{(t,z+z')}).
    \end{align*}
For each $\chi_{2^{\frac{m-e}{2}}}^{(v,w,\epsilon)}\in\tup{Irr}_{(2^{\frac{m-e}{2}})}(G)$, we consider $v=t$ or not separately.
If $v=t$, we obtain from \eqref{eq_nonli} that
     \[
       \omega_{\chi_{2^{\frac{m-e}{2}}}^{(t,w,\epsilon)}}(D_{(t,z)})\overline{\omega_{\chi_{2^{\frac{m-e}{2}}}^{(t,w,\epsilon)}}(D_{(t,z')})}=
      \psi_v(B_{(t,z)})\psi_v(B_{(t,z')})=\eta_z\eta_{z'}2^{2(m-1)}=-2^{m-1}\omega_{\chi_{2^{\frac{m-e}{2}}}^{(t,w,\epsilon)}}(D_{(t,z+z')}).
  \]
     The last equality is deduced from the observation that $\psi_t(B_{t,z+z'})=-\eta_z\eta_{z'}2^{m-1}$.
Note that the complex conjugate of $i^{\kappa(\delta')}$ is equal to $(-1)^{\delta'}i^{\kappa(\delta')}$, and $\kappa(\delta+\delta')=\kappa(\delta)+\kappa(\delta')-2\kappa(\delta\delta')$.
Besides, the definition of $Q(x)$ listed in \eqref{eq_Qx} implied that
\[
Q(u_1+u_1')=Q(u_1)+Q(U_1')+\tup{Tr}_e(u_1u_1').
\]
So, if $v\neq t$, we have
     \begin{align*}
      \omega_{\chi_{2^{\frac{m-e}{2}}}^{(v,w,\epsilon)}}(D_{(t,z)})\overline{\omega_{\chi_{2^{\frac{m-e}{2}}}^{(v,w,\epsilon)}}(D_{(t,z')})} &=
     (-1)^{\delta'}\varepsilon_z\varepsilon_{z'}(\epsilon i)^{\kappa(\delta)+\kappa(\delta')}(-1)^{Q(u_1)+Q(u_1')+\tup{Tr}_e(\delta u_0u_1)+\tup{Tr}_n(\delta' u_0u_1')}\phi_w(u_1+u_1')2^{2(m-1)}\\
      &=-2^{m-1}\varepsilon_{z+z'}(\epsilon i)^{\kappa(\delta+\delta')}(-1)^{Q(u_1+u_1')+\tup{Tr}_e((\delta+\delta') u_0(u_1+u_1'))}\phi_w(u_1+u_1') 2^{m-1}\\
&=-2^{m-1}\omega_{\chi_{2^{\frac{m-e}{2}}}^{(v,w,\epsilon)}}(D_{(t,z+z')}).
     \end{align*}
Thus we complete the proof.
  \end{proof}

\begin{thm}\label{link2}
 For each $z\in\bF_{2^e}^*$, let $D_{z}=\cup_{a\in\bF_{2^m}^*, x\in \Gamma_z}C_{a,x}$, where
 $\Gamma_z=\tup{Ker}(\phi_z)$ or $\bF_{2^e}\setminus\tup{Ker}(\phi_z)$.
 Let $\mathcal{R}=\{ D_z: z\in \bF_{2^e}^*\}$. Then $\mathcal{R}$ is a
reduced linking system of difference sets in $G=A_2(m,\theta)$ with parameters\[(2^{2m},2^{m-1}(2^m-1),2^{m-1}(2^{m-1}-1),2^{2(m-1)}; 2^e-1).\]
\end{thm}
\begin{proof}
 Note that $D_z$ is the central difference set given in Theorem \ref{DS2}.
The proof is similar to that of Theorem \ref{thm_lk1}, so we only give a sketch here.
 For any $z,z'\in\bF_{2^e}^*,z\neq z'$, set \[D_{z+z'}=\cup_{a\in\bF_{2^m}^*, x\in \Gamma_{z+z'}}C_{a,x}, \] where
 \[\Gamma_{z+z'}=
 \begin{cases}
 \tup{Ker}(\phi_{z+z'})&\tup{ if } -\varepsilon_z\varepsilon_{z'}(-1)^{\delta'+\delta\delta'+\tup{Tr}_e(u_1u_1'+\delta' u_0u_1+\delta u_0u_1')}=1,\\
 \bF_{2^e}\setminus\tup{Ker}(\phi_{z+z'})&\tup{ if } -\varepsilon_z\varepsilon_{z'}(-1)^{\delta'+\delta\delta'+\tup{Tr}_e(u_1u_1'+\delta' u_0u_1+\delta u_0u_1')}=-1.
 \end{cases}
 \]
 We show that for the pair  $D_{z},D_{z'}\in \mathcal{R}$, we have \[\omega_{\chi}(D_z)\overline{\omega_{\chi}(D_{z'})}=-2^{m-1}\omega_\chi({D_{z+z'}}).\]
    For any $\chi_1^v\in \tup{Lin}(G)^*$, we obtain from  \eqref{linD1} that
    \begin{equation*}
      \omega_{\chi_1^v}(D_{z})\overline{\omega_{\chi_1^v}(D_{z'})}=-2^{m-1} \omega_{\chi_1^v}(D_{z+z'}).
    \end{equation*}
  Write $\sqrt{z}=z^{2^{e-1}}$. Suppose $\sqrt{z}=\delta u_0+u_1$ and $\sqrt{z'}=\delta' u_0+u_1'$. Set
\[\epsilon_z=
\begin{cases}
1,&\tup{ if } \Gamma_z=\tup{Ker}(\phi_z),\\
-1,&\tup{ if } \Gamma_z=\bF_{2^e}\setminus\tup{Ker}(\phi_z).
\end{cases} \]
Then we have $\varepsilon_{z+z'}=-\varepsilon_z\varepsilon_{z'}(-1)^{\delta'+\delta\delta'+\tup{Tr}_e(u_1u_1'+\delta' u_0u_1+\delta u_0u_1')}$.
    For  $\chi_{2^{\frac{m-e}{2}}}^{(v,w,\epsilon)}\in \tup{Irr}_{(2^{\frac{m-e}{2}})}(G)$, we obtain from  \eqref{nonlinD} that
    \begin{align*}
      \omega_{\chi_{2^{\frac{m-e}{2}}}^{(v,w,\epsilon)}}(D_{z})\overline{\omega_{\chi_{2^{\frac{m-e}{2}}}^{(v,w,\epsilon)}}(D_{z'})}
     &=(-1)^{\delta'}\varepsilon_z\varepsilon_{z'}2^{2(m-1)}(\epsilon i)^{\kappa(\delta)+\kappa(\delta')}(-1)^{Q(u_1)+Q(u_1')+\tup{Tr}_e(\delta u_0u_1)+\tup{Tr}_e(\delta' u_0u_1')}\phi_w(u_1+u_1')\\
   &=-2^{m-1}\varepsilon_{z+z'}2^{(m-1)}(\epsilon i)^{\kappa(\delta+\delta')}(-1)^{Q(u_1+u_1')+\tup{Tr}_e((\delta+\delta') u_0(u_1+u_1'))}\phi_w(u_1+u_1')\\
 &=-2^{m-1}\omega_{\chi_{2^{\frac{m-e}{2}}}^{(v,w,\epsilon)}}({D_{z+z'}}).
     \end{align*}

  \end{proof}
\begin{remark}
   From Theorem 5.4 and  Theorem 5.6 in \cite{linking system2}, we obtain that $G=A_2(m,\theta)$ contains a reduced linking system of difference sets of size $2^m-1$.
   In Theorem \ref{thm_lk1} and Theorem \ref{link2}, we give infinite families of linking systems of central difference sets in $G$ of size $2^e-1$ by using a completely different method.
\end{remark}

\subsection{Central partial difference sets in $G=A_p(m,\theta)$ with $p$ odd}\quad

\begin{thm}\label{PDS1}
Let $G=A_p(m,\theta)$. Suppose that $m$ and $f$ are both odd and $p\nmid f$. Take $T$ to be a set of coset representatives for $\bF_p^*$ in $\bF_{p^m}^*$ such that $T\subseteq\square_{p^m}$.
 For each $t\in T$, take $B_t=\bF_{p^m}\setminus\tup{Ker}(\psi_t)$, where $\psi_t$ is as in \eqref{psi_v}. Choose $z \in \square_{p^e}$ such that $\tup{Tr}_e(z)=0$. Take $x\in\bF_{p^e}$ such that $x^2=z$ and write $\sqrt{z}=x$.
  Set
  \begin{equation}\label{eq_T(a)}
   \Gamma(a)=
  \begin{cases}
   \varnothing, &\tup{ if } a\in a_t\sqrt{z}\bF_p,\\
 \tup{Ker}(\phi_z), & \tup{ if } a\in \bF_{p^m}\setminus a_t\sqrt{z}\bF_p,
  \end{cases}
  \end{equation}
where $\phi_z$ is as in \eqref{phi_w}.
  Suppose
 \[
D_{(t,z)}=\cup_{b\in B_t}C_b\cup_{a\in \bF_{p^m}^*,x\in \Gamma(a)}C_{a,x},
\]where $C_b$ and $C_{a,x}$ are the conjugacy classes defined in Lemma \ref{cc}.
  Then $D_{(t,z)}$ is a regular  Latin square type central partial difference  set in $G$ with parameters \[(p^{2m},p^{m-1}(p^m-1), p^{2(m-1)}+p^m-3p^{m-1},p^{2(m-1)}-p^{m-1}).\]
\end{thm}
\begin{proof}
 Note that
 \[
 |D_{(t,z)}|=p^m-p^{m-1}+p^{m-e}(p^m-p)\cdot p^{e-1}=p^{m-1}(p^m-1).
  \]
From Theorem \ref{char theorem} $(2)$,  we only need to show that \[\omega_{\chi}(D_{(t,z)})\in\{-p^{m-1}, p^{m-1}(p-1)\}\] for all $\chi\in\tup{Irr}(G)^*$.
 Firstly, for each  $\chi_1^v\in\tup{Lin}(G)^*$ listed in Table \ref{tableA(p,odd)}, we have
 \begin{align*}
  \omega_{\chi_1^v}(D_{(t,z)})=&|B_t|+p^{m-e}\sum_{a\in\bF_{p^m}\setminus a_t\sqrt{z}\bF_p}|\Gamma(a)|\psi_v(a)\\
   = &p^m-p^{m-1}-p^{m-1}\psi_v(a_t\sqrt{z}\bF_p).
 \end{align*}
 Noting that
 \[
 \psi_v(a_t\sqrt{z}\bF_p)=
 \begin{cases}
   p,& \tup{ if } a_t\sqrt{z}\in\tup{Ker}(\psi_v),\\
   0,& \tup{ otherwise},
 \end{cases}
 \]
 we get
 \[
 \omega_{\chi_1^v}(D_{(t,z)})=
 \begin{cases}
-p^{m-1},& \tup{ if } a_t\sqrt{z}\in\tup{Ker}(\psi_v),\\
p^{m-1}(p-1), &\tup{ otherwise}.
 \end{cases}
 \]

 Next consider the non-linear irreducible characters of $G$.
Since $T\subseteq\square_{p^m}$, from \eqref{eq_xv} we  have  $x_v=1$ for all $v\in T$.
For each  $\chi_{p^{\frac{m-e}{2}}}^{(v,w,s)}, v\in T, w\in\bF_{p^e}, s\in\{1,2\cdots,p-1\}$, listed in  Table \ref{tableA(p,odd)}, we have
 \begin{align}\label{eq_nolp}
  \omega_{\chi_{p^{\frac{m-e}{2}}}^{(v,w,s)}}(D_{(t,z)})=
  &\psi_{sv}(B_t)+p^{m-e}\sum_{u\in\bF_{p^e}}\xi_p^{-\frac{1}{2}sf\tup{Tr}_e(u^2)}\phi_w(u)\phi_{sf u^2}(\Gamma(a_vu))\notag\\
   = &\psi_{sv}(B_t)+p^{m-e}\sum_{u\in \sqrt{z}\bF_{p}}\xi_p^{-\frac{1}{2}sf\tup{Tr}_e(u^2)}\phi_w(u)\phi_{sf  u^2}(\Gamma(a_vu))\notag\\
    = &\psi_{sv}(B_t)+p^{m-e}\sum_{u\in \sqrt{z}\bF_{p}}\phi_w(u)\phi_{sf  u^2}(\Gamma(a_vu)).
 \end{align}
 The second equality is deduced from the fact that $\phi_{sf u^2}(\Gamma(a_vu))=0$ if $u\notin\sqrt{z}\bF_p$,
and the last equality  is deduced from $\tup{Tr}_e(z)=0$.
When $v=t$, we have
 \[
\psi_{st}(B_t)=-p^{m-1},\,\sum_{u\in \sqrt{z}\bF_{p}}\phi_w(u)\phi_{sfu^2}(\Gamma(a_vu))=0.
\]
So it deduces that
$\omega_{\chi_{p^{\frac{m-e}{2}}}^{(t,w,s)}}(D_{(t,z)})=-p^{m-1}$.
When  $v\neq t$, we have $\psi_{sv}(B_t)=0$.
We claim that \[a_v\sqrt{z}\notin a_t\sqrt{z}\bF_p.\]
If not, suppose $a_v=sa_t$ for some $s\in\bF_p^*$. Then $a_v\theta(a_v)=s^2a_t\theta(a_t)$. Since $T \subseteq\square_{\bF_{p^m}}$, we have $a_v\theta(a_v)=v^{-1}$ and $a_t\theta(a_t)=t^{-1}$. Thus it follows that \[tv^{-1}=s^2\in\bF_p^*, \tup{ i.e.},\, t\in\bF_p^*v.\] Since $T$ is a set of coset representatives for $\bF_p^*$ in $\bF_{p^m}^*$, we get a  contradiction.
The claim above implies that \[\Gamma(a_vu)=\tup{Ker}(\phi_z)\tup{ for each }u\in \sqrt{z}\bF_{p}^*.\]
So we have when $v\neq t$,
\begin{align*}
          \eqref{eq_nolp}=& p^{m-e}\sum_{u\in \sqrt{z}\bF_{p}^*}\phi_w(u)\phi_{sf  u^2}(\tup{Ker}(\phi_z))
           = p^{m-1}\phi_w(\sqrt{z}\bF_{p}^*)  \\
           =&
      \begin{cases}
      (p-1)p^{m-1}, &\tup{ if } \sqrt{z}\in\tup{Ker}(\phi_w),\\
      -p^{m-1},&\tup{ if } \sqrt{z}\notin\tup{Ker}(\phi_w).
      \end{cases}
          \end{align*}
Thus for any non-linear irreducible character $\omega_{\chi_{p^{\frac{m-e}{2}}}^{(v,w,s)}},v\in T, w\in\bF_{p^e}, s\in\{1,2\cdots,p-1\}$, we have
 $\omega_{\chi_{p^{\frac{m-e}{2}}}^{(v,w,s)}}(D_{(t,z)})\in\{-p^{m-1}, p^{m-1}(p-1)\}$,
and we complete the proof.
\end{proof}
\begin{remark}
  The condition that $m$ is odd in Theorem \ref{PDS1} can not be removed.
 Suppose that $\gamma$ is a primitive element of $\bF_{p^m}^*$, then $\bF_p^*=\langle \gamma^{\frac{p^m-1}{p-1}} \rangle$. If $m$ is odd, then $\frac{p^m-1}{p-1}$ is odd. So,
 \[
\bF_p^*\cap\square_{p^m}\neq \varnothing,\,\tup{ and } \bF_p^*\cap\blacksquare_{p^m}\neq \varnothing.
 \]
Thus we can
take $T$ to be a set of coset representatives for $\bF_p^*$ in $\bF_{p^m}^*$ such that $T\subseteq\square_{\bF_{p^m}}$.
\end{remark}
\begin{thm}\label{PDS2}
Let $G=A_p(m,\theta)$. Suppose that $m$ and $f$ are both odd and $p\nmid f$.   Choose $z \in \square_{\bF_{p^e}}$ such that $\tup{Tr}_e(z)=0$. Take $\Gamma_z=\tup{Ker}(\phi_z)$, where $\phi_z$ is as in \eqref{phi_w}. Let
  \[D_{z}=\cup_{a\in\bF_{p^m}^*, x\in \Gamma_z}C_{a,x},\,D_z'=\cup_{b\in\bF_{p^m}^*}C_b\cup_{a\in\bF_{p^m}^*,x\in\Gamma_z}C_{a,x},\]
where $C_b$ and $C_{a,x}$ are the conjugacy classes defined in Lemma \ref{cc}.
  Then  $D_z$ is a Latin square type  central partial difference  set in $G$ with parameters \[(p^{2m},p^{m-1}(p^m-1), p^{2(m-1)}+p^m-3p^{m-1},p^{2(m-1)}-p^{m-1}),\]
and $D_z'$  is a Latin square type  central partial difference  set  with parameters \[(p^{2m},(p^{m-1}+1)(p^m-1),p^{2(m-1)}+p^m-p^{m-1}-2,p^{2(m-1)}+p^{m-1} ).\]
\end{thm}
\begin{proof}
Since the proof that $D_z'$ is a partial difference set in $G$ is similar to the proof that $D_z$ is a partial difference set in $G$, we only give the details for $D_z$.  Noting that $m$ is odd, we can take $T$ to be a set of coset representatives for $\bF_p^*$ in $\bF_{p^m}^*$ such that $T\subseteq\square_{p^m}$.
 Note that \[|D_{z}|=p^{m-e}(p^m-1)p^{e-1}=p^{m-1}(p^m-1).\]
Similar to the proof of Theorem \ref{PDS1}, from Theorem \ref{char theorem}  we only need to show that  \[\omega_{\chi}(D_{(t,z)})\in\{-p^{m-1}, p^{m-1}(p-1)\} \tup{ for all }\chi\in\tup{Irr}(G)^*.\]
First we have \[\omega_{\chi_1^v}(D_z)=p^{m-1}\psi_v(\bF_{p^m}^*)=-p^{m-1}\tup{ for any }v\in\bF_{p^m}^*.\]
Take $\sqrt{z}$ to be an element in $\bF_{p^e}$ such that its square is $z$. Then we obtain that for any $v\in T, w\in\bF_{p^e},s\in\{1,\cdots,p-1\}$,
\[\omega_{\chi_{p^{\frac{m-n}{2}}}^{(v,w,s)}}(D_{z})=p^{m-e}\sum_{u\in\sqrt{z}\bF_{p}^*}\phi_w(u)p^{e-1} =
\begin{cases}
  (p-1)p^{m-1}, &\tup{ if } \sqrt{z}\in\tup{Ker}(\phi_w),\\
-p^{m-1},&\tup{ if } \sqrt{z}\notin\tup{Ker}(\phi_w).
\end{cases}\]
Hence $D_z$ is a partial difference set in $G$.
\begin{remark}
Take conditions in Theorem \ref{PDS2}.
Let \[D_z''=\cup_{a\in\bF_{p^m}^*,x\in\Gamma'(z)}C_{a,x},\] where $\Gamma'(z)=\bF_{p^m}\setminus\tup{Ker}(\phi_z)$. Then $D_z''$ is also a partial difference set in $G$, and the proof is similar to that of $D_z$.
 Besides, $D_z'$ given in Theorem \ref{PDS2} can be obtained by taking complement of $D_z''$ and subtracting $1_G$.
\end{remark}
\end{proof}

\section{Concluding remark}
 In the case $G=A_2(m,\theta)$, we show that there are no non-trivial central difference sets in $G$ when $o(\theta)$ is even. However, when $o(\theta)>1$ is odd, we give two distinct constructions of central difference sets. Particularly, when $o(\theta)=m$, those two constructions give all non-trivial central difference sets in $G$. Besides, we  construct some linking systems of central difference sets in $G$ by using the difference sets we constructed.
 In the case $G=A_p(m,\theta)$ with $p>2$,  when $o(\theta)$ is even, we show that there are no non-trivial Latin square type central partial difference set in $G$ with $r=p^{m-1}$ or $p^{m-1}+1$.  When $o(\theta)>1$ is odd, we extend our two  constructions of central difference sets in $A_2(m,\theta)$ to  $G$ and obtain two infinite families of Latin square type central partial difference sets in $G$.

 \section*{Acknowledgement}
We are grateful to Prof. Tao Feng for many useful suggestions. The work of the two authors was supported by National Natural Science Foundation of China under Grant No. 11771392.

\end{document}